\documentclass[12pt,a4paper]{amsart}
\usepackage{amsmath}
\usepackage{amsfonts}
\usepackage{amssymb}
\usepackage{amsthm}
\usepackage[utf8]{inputenc}
\usepackage{array}
\usepackage{fancyhdr}
\usepackage{geometry}
\usepackage{multicol}
\usepackage{enumerate}
\usepackage{stackrel}
\usepackage{color}
\usepackage{hyperref}

\newcommand{\B}{\mathbf{B}}		

\newcommand{\id}{\operatorname{id}}

\newcommand{\Tan}{\operatorname{Tan}}

\renewcommand{\le}{\leqslant}
\renewcommand{\ge}{\geqslant}

\newcommand{\cH}{\mathcal{H}}

\newcommand{\cN}{{\ensuremath{\mathcal N}}}

\newcommand{\Sph}{{\ensuremath{\mathbb S}}}

\newcommand{\R}{{\ensuremath{\mathbb R}}}

\newcommand{\sing}{\operatorname{sing}}

\newcommand{\dd}{\phantom{.}\mathrm{d}}
\newcommand{\loc}{\mathrm{loc}}

\definecolor{czerwony}{rgb}{0.9, 0.2, 0.1}

\def\XXint#1#2#3{{\setbox0=\hbox{$#1{#2#3}{\int}$}
     \vcenter{\hbox{$#2#3$}}\kern-.5\wd0}}

\numberwithin{equation}{section}

\newtheorem{thm}{Theorem}[section]
\newtheorem{lem}[thm]{Lemma}
\newtheorem{cor}[thm]{Corollary}
\newtheorem{prop}[thm]{Proposition}

\theoremstyle{definition}

\newtheorem{rem}[thm]{Remark}

\newtheorem{ex}[thm]{Example}

\newtheorem{df}[thm]{Definition}

\binoppenalty=\maxdimen
\relpenalty=\maxdimen

\author[M. Miśkiewicz]{Michał Miśkiewicz}
\title[Regularity of the singular set]{On H{\"o}lder regularity of the singular set of energy minimizing harmonic maps into closed manifolds}
\address{Institute of Mathematics, University of Warsaw,\newline Banacha 2, 02-097 Warszawa, Poland}
\email{m.miskiewicz@mimuw.edu.pl}
\thanks{The research has been supported by the NCN grant no. 2016/21/B/ST1/03138 (years 2017-2020).}


\begin{document}

\begin{abstract}
Energy minimizing harmonic maps between manifolds are known to be smooth outside a rectifiable set of codimension 3, called the singular set. The possibility that this set is not a manifold, but has arbitrarily many small gaps in it, is not excluded in general. Here we prove that some part of the singular set -- characterized by topological and analytic properties of tangent maps -- is a topological manifold. In the special case of maps into the sphere ${\mathbb S}^2$, we conclude that the whole top-dimensional part of the singular set is a manifold -- this generalizes a similar result in two-dimensional domain, due to Hardt and Lin. 
\end{abstract}

\maketitle

\section{Introduction}

\subsection*{Singularities of energy minimizing harmonic maps}

Energy minimizing harmonic maps between manifolds may have singularities if the domain dimension is $3$ or higher. The most well-known example is the map 
\[
\R^3 \times \R^{n-3} \ni (x,y) \mapsto x/|x| \in \Sph^2. 
\]
In general, any energy minimizer $u$ is smooth outside the closed singular set $\sing u$ of Hausdorff dimension $n-3$ or less, $n$ being the dimension of the domain (Schoen, Uhlenbeck \cite{SchUhl82,SchUhl83cor}). The phenomenon of singularities is now well-understood in dimension $3$, when singularities form a discrete set. In recent years, there has been a~substantial progress concerning the case $n \ge 4$. Naber and Valtorta \cite{NabVal17} have proved that the singular set has locally finite $(n-3)$-dimensional Hausdorff measure and is $(n-3)$-rectifiable, i.e., can be essentially covered by countably many Lipschitz images of $\R^{n-3}$; the latter was already known (due to Simon \cite{Sim95}) in the case when the target manifold is real-analytic. 

The results cited above are mostly concerned with the size of the singular set, but do not imply \textit{lower bounds} on the singular set. In particular, the possibility that the singular set is an arbitrary subset of an $(n-3)$-dimensional manifold (with many small gaps) is not excluded by \cite{NabVal17,SchUhl82,SchUhl83cor,Sim95}. 

Lower bounds on the size are indeed possible in the presence of a~topological obstruction; the following example is simple but instructive. 

\begin{ex}
Consider the smooth boundary map $\varphi \colon \Sph^2 \times \Sph^1 \to \Sph^2$ given by $\varphi(x,y) = x$ and (some) $u \colon \B^3 \times \Sph^1 \to \Sph^2$ minimizing the energy in the class of maps equal to $\varphi$ on the boundary. Restricting $u$ to a slice $\B^3 \times \{ y \}$ and applying Brouwer's theorem, we see that each such slice contains a singular point. This shows that $\cH^1(\sing u) \ge \cH^1(\Sph^1) = 2 \pi$. In this particular case one can actually prove that $u(x,y)=x/|x|$, but the presented reasoning applies also to any $\varphi$ in the same homotopy class. 
\end{ex}

In the special case of maps $u \colon \B^4 \to \Sph^2$, Hardt and Lin \cite{HarLin90} obtained the following remarkable result. 
\begin{thm}
\label{thm:HarLin90}
The singular set of an energy minimizer $u \colon \B^4 \to \Sph^2$ is locally a~union of a~finite set and a~finite family of H{\"o}lder continuous closed curves with a~finite number of crossings. 
\end{thm}
The same claim was obtained also for maps $u \colon \B^5 \to \Sph^3$ (Lin-Wang \cite{LinWan06}). To the author's knowledge, these are the only two cases where $\sing u$ was shown to be essentially a~manifold.  

The above theorem relies on the classification of tangent maps from $\R^3$ into $\Sph^2$ carried out by Brezis, Coron and Lieb \cite{BreCorLie86}; for $\Sph^3$, a similar classification was obtained by Nakajima \cite{Nak06}. These maps describe the infinitesimal behavior of $u$ at a typical point of $\sing u$. 

\subsection*{Main results}

The present paper aims to extract the topological obstruction responsible for preventing gaps in the singular set of maps into $\Sph^2$. To this end, we distinguish particular homotopy classes of tangent maps $\R^3 \to \cN$ (called here \textit{indecomposable classes}) for any closed Riemannian manifold $\cN$. 

To each homotopy class $\alpha \in \pi_2(\cN)$ we assign its lowest energy level $\Theta(\alpha)$ and call $\alpha$ indecomposable if $\Theta(\alpha) < \infty$ and $\alpha$ cannot be represented as a sum of homotopy classes $\alpha_j \in \pi_2(\cN)$ with strictly smaller energy levels $\Theta(\alpha_j)$. We then restrict our attention to singularites with fixed topological type $\alpha$ -- we define $\sing_\alpha u$ to be the set of points at which some tangent map of $u$ has type $\alpha$. Rigorous definitions are given in Section~\ref{ch:HCMs}. 

Another goal is to generalize the result of Hardt and Lin \cite{HarLin90} to higher dimensional domains. The difficulty lies in the fact that the singular set is stratified -- it decomposes into parts of different dimensions. For $u \colon \B^4 \to \Sph^2$, there are only two strata: one is formed by H{\"o}lder continuous curves and the other by their crossing points and some isolated points. In the theorem below, we were only able to study the top-dimensional part $\sing_* u$ of the singular set. Again, the necessary notions are introduced in Section \ref{ch:tangent-maps}. 

For simplicity, we only consider the standard Euclidean ball $\B^n$ as the domain, but the results hold true for any manifold. This is due to the fact that we only consider the infinitesimal behavior of maps. A detailed explanation can be found in \cite{NabVal17} and \cite[Sec.~8]{Sim95}. 

\begin{thm}
\label{thm:Holder-regularity}
Let $u \colon \B^n \to \cN$ be an energy minimizing map into a closed Riemannian manifold $\cN$, $\alpha \in \pi_2(\cN)$ be an indecomposable homotopy class, and $\Theta(\alpha)$ be its lowest energy level. Then for each exponent $0 < \gamma < 1$ there is $\delta(\gamma,n,\alpha,\cN) > 0$ such that the set 
\[
\left \{
x \in \sing_\alpha u : 
\lim_{r \to 0} r^{2-n} \int_{\B_r(x)} |\nabla u|^2 < \Theta(\alpha) + \delta
\right \}
\]
forms an open subset of $\sing u$ and it is a topological $(n-3)$-dimensional manifold of H{\"o}lder class $C^{0,\gamma}$. 
\end{thm}

In the case when $\cN$ is a real-analytic manifold, Simon \cite[Lemma~4.3]{Sim95} showed that the set of possible energy densities $\lim_{r \to 0} r^{2-n} \int_{\B_r(x)} |\nabla u|^2$ is discrete. This allows us to slightly strenghten the statement above. The same conclusion holds also if $\cN$ satisfies the integrability assumption introduced in \cite[Ch.~3.13]{Sim96}. 

\begin{cor}
\label{cor:discrete-levels}
If $u \colon \B^n \to \cN$ is an energy minimizing map into a real-analytic manifold $\cN$ and $\alpha \in \pi_2(\cN)$ is an indecomposable homotopy class, then 
\[
\left \{
x \in \sing_\alpha u : 
\lim_{r \to 0} r^{2-n} \int_{\B_r(x)} |\nabla u|^2 = \Theta(\alpha) 
\right \}
\]
forms an open subset of $\sing u$ and it is a topological $(n-3)$-dimensional manifold of H{\"o}lder class $C^{0,\gamma}$ with any $0 < \gamma < 1$. 
\end{cor}

Specializing to the case $\cN = \Sph^2$ and recalling the classification of tangent maps \cite{BreCorLie86}, we obtain a partial generalization of Theorem \ref{thm:HarLin90} \cite{HarLin90} to arbitrary dimensions: 

\begin{cor}
\label{cor:sphere-target}
If $u \colon \B^n \to \Sph^2$ is an energy minimizing map, then the top-dimensional part $\sing_* u$ forms an open subset of $\sing u$ and it is a topological $(n-3)$-dimensional manifold of H{\"o}lder class $C^{0,\gamma}$ with any $0 < \gamma < 1$. 
\end{cor}

\subsection*{An outline}

Section \ref{ch:preliminaries} recalls all notions and results needed in the sequel. These mostly come from the seminal work of Schoen and Uhlenbeck \cite{SchUhl82}, but the presentation here follows Simon's lecture notes \cite{Sim96}. We also define indecomposable homotopy classes of maps from $\Sph^2$ into $\cN$. 

Since Theorem \ref{thm:Holder-regularity} only concerns the singularities of indecomposable types, it is worthwhile to investigate the existence of such classes, which we do in Section \ref{ch:good-classes-exist}. Indeed, we show that for any $\cN$ the second homotopy group $\pi_2(\cN,p)$ is generated (up to the action of $\pi_1(\cN,p)$ on $\pi_2(\cN,p)$) by indecomposable homotopy classes. This is very close to the classical (slightly weaker) result due to Sacks and Uhlenbeck \cite{SacUhl81} (see also \cite{Str85}): smooth harmonic maps from $\Sph^2$ into $\cN$ generate the whole group $\pi_2(\cN,p)$ (up to the action of $\pi_1(\cN,p)$). 

To obtain bi-H{\"o}lder-equivalence with a Euclidean ball, we employ Reifenberg's topological disc theorem \cite{Rei60} (see also~\cite{Sim-notes}). We recall its statement and the so-called Reifenberg flatness condition in Section \ref{ch:flatness}. We also introduce a flatness condition for an energy minimizer $u$ which includes Reifenberg flatness for $\sing u$, but also forces $u$ to be close to a tangent map. 

The main results are proved in Section \ref{ch:proof}. The difficulty in applying Reifenberg's theorem to $\sing u$ lies in showing that this set has no gaps. This is done in Lemma \ref{lem:no-drop}; this is also the point where our topological assumptions play a role. Then we are able to show that if $u$ satisfies the flatness condition on the ball $\B_2(0)$, it also satisfies the same condition on each smaller ball $\B_r(0)$ (Corollary \ref{cor:all-scales}) and on each ball $\B_1(z)$ centered at a point $z \in \B_1$ with enough energy density (Proposition \ref{prop:all-balls}). Combining these results, we check the hypotheses of Reifenberg's theorem and establish Theorem \ref{thm:Holder-regularity}.

Some interesting observations not needed for the proof of Theorem \ref{thm:Holder-regularity} are gathered in Section \ref{ch:additional}. 

\section{Preliminaries}
\label{ch:preliminaries}

\subsection*{Regularity of energy minimizers}

In what follows, $u \colon \B^n \to \cN$ is an energy minimizing map into a closed Riemannian manifold $\cN$. Here we recall the basic properties of such maps \cite{SchUhl82}. 

Denote the rescaled energy 
\[
\theta_u(x,r) := r^{2-n} \int_{\B_r(x)} |\nabla u|^2 
\quad \text{for } \B_r(x) \subseteq \B_1, 
\]
which is monotone in $r$: 
\[
\tfrac{\partial}{\partial r} \theta_u(x,r) = 2 \int_{\partial \B_r(x)} \frac{|\nabla u \cdot (y-x)|^2}{|y-x|^n} \ge 0.
\]
This enables us to define the energy density at $x$: 
\[
\theta_u(x,0) := \lim_{r \to 0} \theta_u(x,r),
\]
which is by definition an upper semicontinuous function (in both $x \in \B^n$ and $u \in W^{1,2}$) \cite[2.11]{Sim96}. Obviously, $\theta_u(x,0)=0$ at regular points. 

The main regularity statement is the following $\varepsilon$-regularity theorem: 
\begin{align}
\label{eq:eps-regularity}
\text{there is } \varepsilon(n,\cN)>0 \text{ s.t. } \theta_u(x,2r) < \varepsilon 
& \Rightarrow u \text{ is smooth on } \B_r(x), \\
\nonumber
\text{in particular } \theta_u(x,0) < \varepsilon & \Rightarrow x \notin \sing u.
\end{align}

We also note two compactness theorems for a sequence $u_k$ of energy minimizers: 
\begin{itemize}
\item
if $u_k \rightharpoonup u$ in $H^1$, then $u$ is an energy minimizer and the convergence is actually strong in $H^1$ \cite{Luc88}, 
\item
if $u_k \to u$ as above, then the convergence is uniform on compact sets disjoint from $\sing u$ \cite{SchUhl82}. 
\end{itemize}

\subsection*{Tangent maps}
\label{ch:tangent-maps}

Given an energy minimizer $u \colon \B^n \to \cN$ and a point $x \in \B^n$, consider the family of rescaled maps $u_r(y) = u(x+ry)$. By the results from the previous section, each sequence $r_j \to 0$ has a subsequence for which $u_{r_j}$ converges in $W^{1,2}_\loc(\R^n)$ to some energy minimizer $\varphi$, called a tangent map of $u$ at $x$ (possibly non-unique). This limit map is homogeneous, i.e., $\varphi(\lambda x) = \varphi(x)$ for all $\lambda > 0$, $x \in \R^n$. 

For a homogeneous energy minimizer $\varphi \colon \R^n \to \cN$, the energy density $\theta_\varphi(y,0)$ is maximal at $y=0$; moreover, equality $\theta_\varphi(y) = \theta_\varphi(0)$ at some other point $y$ leads to higher symmetry: $\varphi(x+ty) = \varphi(x)$ for all $t \in \R$, $x \in \R^n$. Let $S(\varphi)$ be defined by 
\[
S(\varphi) = \left\{ y \in \R^n : \theta_\varphi(y) = \theta_\varphi(0) \right\}. 
\]
Then $S(\varphi)$ is a linear subspace of $\R^n$ describing the symmetries of $\varphi$: 
\[
\varphi(x+y) = \varphi(x) \quad \text{for all } x \in \R^n, \ y \in S(\varphi). 
\]
For non-constant $\varphi$, we have $S(\varphi) \subseteq \sing \varphi$. If $\dim S(\varphi) = n-3$, this is necessarily an equality. 

\subsection*{Top-dimensional part of the singular set}
\label{ch:HCMs}

If $u$ is an energy minimizer, for each $j = 0,1,2,\ldots,n-1$ we define 
\[
S_j = \left\{ y \in \sing u : \dim S(\varphi) \le j \text{ for all tangent maps $\varphi$ of $u$ at $y$} \right\},
\]
which leads to the classical stratification of the singular set 
\[
S_0 \subseteq S_1 \subseteq \ldots \subseteq S_{n-3} = S_{n-2} = S_{n-1} = \sing u.
\]
It is known \cite{SchUhl82} that each $S_j$ has Hausdorff dimension at most $j$, in particular 
\[
\dim_H \sing u \le n-3.
\] 

\medskip

In what follows, we are interested in the top-dimensional part of the singular set: 
\begin{align*}
\sing_* u & = S_{n-3} \setminus S_{n-4} \\
& = \left\{ y \in \sing u : \dim S(\varphi) = n-3 \text{ for some tangent map $\varphi$ of $u$ at $y$} \right\}. 
\end{align*}
Note that 
\[
\dim_H (\sing u \setminus \sing_* u) \le n-4. 
\]
Following \cite{Sim96}, we shall call any homogeneous energy minimizing $\varphi \colon \R^n \to \cN$ with $\dim S(\varphi) = n-3$ a \textit{homogeneous cylindrical map} (abbreviated HCM). 

\medskip

Consider now such a HCM $\varphi_0 \colon \R^n \to \cN$ with $S(\varphi_0) = \R^{n-3} \times \mathbf{0}$. This map actually depends only on $3$ variables, i.e. $\varphi_0(x,y) = \varphi_1(y)$ for some homogeneous $\varphi_1 \colon \R^3 \to \cN$. By \cite[Lemma~2.1]{HarLin90}, the map $\varphi_1$ defined in this way is energy minimizing if and only if $\varphi_0$ is. Since $\varphi_1$ is homogeneous, it is uniquely determined by its restriction to the unit sphere $\varphi_2 \colon \Sph^2 \to \cN$, which is a smooth harmonic map. 

From now on, we shall abuse the notation and use the same symbol for all three maps $\varphi_0,\varphi_1,\varphi_2$; the precise meaning should be clear from the context. Note that their energies differ by a multiplicative constant: 
\[
\int_{\Sph^2} |\nabla \varphi_2|^2 = \int_{\B_1^3} |\nabla \varphi_1|^2 = C(n) \int_{\B_1^n} |\nabla \varphi_0|^2, 
\]
so energy comparison does not lead to confusion. 

\textit{Homotopy type} of a HCM always refers to the map $\varphi_2 \colon \Sph^2 \to \cN$ (as $\varphi_0,\varphi_1$ are discontinuous and defined on contractible domains). For a general HCM $\varphi_0$ we may choose a rotation $q$ that maps $S(\varphi_0)$ to $\R^{n-3} \times \mathbf{0}$ and thus reduce to the previous case. We then say that $\varphi_0$ has homotopy type $\alpha$ if $\varphi_0 \circ q^{-1}$ restricted to $\mathbf{0} \times \Sph^2$ has type $\alpha$. 

\begin{rem}
There is a subtle ambiguity here. Depending on the choice of $q$, we may obtain two homotopy types that differ by a composition with the antipodal map, i.e. both $[\varphi_2(x)]$ and $[\varphi_2(-x)]$. 
\end{rem}

Using this terminology, singular points in $\sing_* u$ can be classified according to their energy density and the homotopy type of a tangent map. Since we only consider basepoint-free homotopies, we denote by $\pi_2(\cN)$ the set of homotopy classes of continuous maps $\Sph^2 \to \cN$. Note that in general it does not carry a group structure, as it is the quotient of the action of $\pi_1(\cN,p)$ on $\pi_2(\cN,p)$. 
\begin{df}
\label{df:notation}
For any homotopy type $\alpha \in \pi_2(\cN)$ we let 
\[
\sing_{\alpha} u = 
\left\{ y \in \sing u : \text{ some tangent map of $u$ at $y$ is a HCM of type } \alpha \right\}. 
\]
We also denote its lowest energy level by 
\[
\Theta(\alpha) := \inf \left \{ \int_{\B_1^n} |\nabla \varphi|^2 : \varphi \text{ is a HCM of type } \alpha \right \}.
\]
A simple compactness argument shows that this infimum is either infinite (if no HCM has type $\alpha$) or achieved by some minimal HCM. We also let 
\[
\sing_{\ge \Theta} u = 
\left\{ y \sing u : \theta_u(y,0) \ge \Theta \right\}. 
\]
which is a closed set by upper semicontinuity of $\theta_u(\cdot,0)$. 
\end{df}

At this point we cannot exclude the case when there are many homotopically different tangent maps at one point. However, this cannot happen under an additional assumption described below (see Remark \ref{rem:unique-type}). Again, since $\pi_2(\cN)$ is not necessarily a group, the decomposition in Definition \ref{df:decomposition} is to be understood up to the action of $\pi_1(\cN)$, as described in Section \ref{ch:good-classes-exist} (see also the formulation of \cite[Thm.~5.9]{SacUhl81}). 

\begin{df}
\label{df:decomposition}
Consider $\alpha \in \pi_2(\cN)$ with $\Theta(\alpha) < \infty$. This homotopy class is called decomposable if there is a decomposition 
\[
\alpha = \alpha_1 + \ldots + \alpha_k \quad \text{in } \pi_2(\cN),
\]
where $\Theta(\alpha_j) < \Theta(\alpha)$ for each $j=1,\ldots,k$. Otherwise $\alpha$ is called indecomposable. 
\end{df}

Note that the above criterion does not depend on the dimension $n$, but only on the manifold $\cN$. 

As a special case, $\alpha$ is indecomposable if $\Theta(\alpha)$ is the smallest among all non-trivial homotopy types. In this case the proof is much easier (see the remark below Lemma \ref{lem:no-drop}).

\medskip

Similar decompositions of this type appear naturally as a result of the \textit{bubbling phenomenon} when one tries to minimize the energy in a given homotopy class. More precisely, recall that by \cite{SacUhl81} (see also \cite{Str85}) any smooth map $\varphi \colon \Sph^2 \to \cN$ can be decomposed as $[\varphi] = [\varphi_1] + \ldots + [\varphi_k]$, where each $\varphi_j$ is a harmonic map and 
\[
\sum_{j=1}^k \int_{\Sph^2} |\nabla \varphi_j^2| \le \int_{\Sph^2} |\nabla \varphi^2|.
\]
Motivated by these decompositions, one could replace the condition $\Theta(\alpha) > \max_j \Theta(\alpha_j)$ in Definition \ref{df:decomposition} by $\Theta(\alpha) \ge \sum_j \Theta(\alpha_j)$, thus enlarging the set of indecomposable classes. 
A natural conjecture here would be that Theorem \ref{thm:Holder-regularity} continues to hold in this case, but the author was not able to verify it. 

\begin{ex}
By the classification from \cite{BreCorLie86}, the only HCMs into the sphere $\Sph^2$ are isometries $\varphi \colon \Sph^2 \to \Sph^2$. Thus, for $\alpha \in \pi_2(\Sph^2)$ we have 
\[
\Theta(\alpha) = 
\begin{cases}
0 & \text{ for } \alpha = 0, \\
4 \pi & \text{ for } \alpha = [\pm \id], \\
\infty & \text{ otherwise.}
\end{cases}
\]
By Definition \ref{df:decomposition}, the indecomposable classes here are $0,[\id],[-\id]$. Note that these classes generate the whole group $\pi_2(\Sph^2)$ (see Proposition \ref{prop:indecomposable-types} for the general case). 
\end{ex}

\section{Existence of indecomposable homotopy classes}
\label{ch:good-classes-exist}

We show that the set of all indecomposable homotopy classes generates $\pi_2(\cN)$. Similarly to \cite[Thm.~5.9]{SacUhl81}, we only consider basepoint-free homotopies, so this statement should be understood as generating $\pi_2(\cN,p)$ up to the action of $\pi_1(\cN,p)$. In other words, for any $\alpha \in \pi_2(\cN)$ there are indecomposable homotopy classes $\alpha_1,\ldots,\alpha_k \in \pi_2(\cN)$ and a continuous map 
\[
u \colon \B^3 \setminus \bigcup_{j=1}^k \B_j \to \cN 
\]
such that $u|_{\partial \B} \in \alpha$ and $u|_{\partial \B_j} \in \alpha_j$, where $\B_j \Subset \B$ are smaller disjoint balls. 

\medskip

This can be divided into two steps as follows. 

\begin{prop}
\label{prop:indecomposable-types}
Let $\cN$ be a closed Riemannian manifold. Then 
\begin{enumerate}[(a)]
\item
the set of all HCMs $\varphi \colon \Sph^2 \to \cN$ generates $\pi_2(\cN)$, 
\item
each HCM $\varphi \colon \Sph^2 \to \cN$ as an element of $\pi_2(\cN)$ can be decomposed into indecomposable homotopy classes. 
\end{enumerate}
\end{prop}

\begin{proof}
To show part (a), fix $\alpha \in \pi_2(\cN)$ and choose a smooth map $\varphi \colon \Sph^2 \to \cN$ of this type. Then there exists (possibly non-unique) $u \in W^{1,2}(\B_1^3,\cN)$ such that 
\[
\int_{\B_1} |\nabla u|^2 = \min \left \{ \int_{\B_1} |\nabla v|^2 : v \in W^{1,2}(\B_1,\cN), \ v = \varphi \text{ on } \Sph^2 \right \}.
\]
Such a minimizer has at most a~finite number of interior singularities $p_1,\ldots,p_k \in \B_1$. At each $p_j$ there is a~(possibly non-unique) tangent map $\varphi_j$, which is necessarily a~HCM; by uniform convergence away from the singularity, $u$ restricted to $\partial \B_r(p_j)$ is homotopic to $\varphi_j$ for some arbitrary small $r$ (in consequence, also for all sufficiently small $r$). This yields the decomposition 
\[
[\varphi] = [\varphi_1] + \ldots + [\varphi_k] \quad \text{ in } \pi_2(\cN). 
\]

Part (b) follows from the definition by a compactness argument, which allows us to exclude infinite decompositions. Consider any homotopy type $\alpha \in \cN$ represented by a HCM, i.e. with $\Theta(\alpha) < \infty$. First let us show that there are only finitely many homotopy types $\beta \in \pi_2(\cN)$ with $\Theta(\beta) \le \Theta(\alpha)$. Indeed, otherwise we would have an infinite sequence of HCMs $\varphi_k \colon \B_1^3 \to \cN$ with distinct homotopy types and uniformly bounded energy. Without loss of generality, $\varphi_k$ converges to some HCM $\varphi$ in $W^{1,2}(\B_1)$, but also in $C^0(\Sph^2)$. This shows that almost all $\varphi_k$ have the same homotopy type as $\varphi$, which is a contradiction. 

If $\alpha$ is decomposable, we have $\alpha = \alpha_1 + \ldots + \alpha_k$, where $\Theta(\alpha_j) < \Theta(\alpha)$ for each $j$. Decomposing further each $\alpha_j$ whenever possible, and iterating this procedure until all obtained homotopy types are indecomposable, we arrive at claim (b). One only needs to note that this procedure stops after at most $N$ steps, where $N$ is the number of homotopy types from the last paragraph. Indeed, any branch of the decomposition tree is a sequence $\beta_0, \beta_1, \beta_2, \ldots$ with $\beta_0 = \alpha$ and $\Theta(\beta_{j+1}) < \Theta(\beta_j)$ for each $j$, so it contains at most $N$ elements. 
\end{proof}

We remark that a similar decomposition was obtained by Sacks and Uhlenbeck \cite{SacUhl81} (see also \cite{Str85}): smooth harmonic maps from $\Sph^2$ into $\cN$ generate the whole group $\pi_2(\cN,p)$ up to the action of $\pi_1(\cN,p)$. It may be that some homotopy classes in $\pi_2(\cN)$ do not contain any harmonic map. Since here we only consider harmonic maps $\varphi \colon \Sph^2 \to \cN$ for which the homogeneous extension $\varphi \colon \B^3 \to \cN$ is energy minimizing, our result is a slight generalization. 

\section{Notions of flatness and Reifenberg's topological disc theorem}
\label{ch:flatness}

H{\"o}lder regularity of the singular set will be obtained by an application of Reifenberg's topological disc theorem \cite{Rei60} (see also~\cite{Sim-notes}). To state it, we first need the following notion of flatness (for our purposes restricted to codimension $3$). 

\begin{df}
\label{df:set-flatness}
A set $A \subseteq \R^n$ is said to be $\varepsilon$-Reifenberg flat in the ball $\B_r(x)$ (with respect to $L$) if 
\[
A \cap \B_r(x) \subseteq \B_{r \varepsilon} L 
\quad \text{and} \quad
L \cap \B_r(x) \subseteq \B_{r \varepsilon} A 
\]
for some $(n-3)$-dimensional affine plane $L$ through $x$.
\end{df}

The above condition means exactly that the normalized Hausdorff distance on $\B_r(x)$ from $A$ to some $(n-3)$-dimensional affine plane through $x$ is not larger than~$\varepsilon$. 

\begin{thm}[Reifenberg's topological disc theorem]
\label{thm:Reifenberg}
For each H{\"o}lder exponent $0 < \gamma < 1$ there is $\varepsilon(n,\gamma) > 0$ such that the following holds. If a closed set $A \subseteq \R^n$ containing the origin is $\varepsilon$-Reifenberg flat in each ball $\B_r(x)$ with $x \in A \cap \B_1$ and $r < 1$, then the set $A \cap \B_1$ is bi-H{\"o}lder equivalent to the closed unit ball $\B^{n-3} \subseteq \R^{n-3}$ with exponent $\gamma$. 
\end{thm}

\medskip

We shall also make repeated use of the following condition for energy minimizing maps. 

\begin{df}
\label{df:flatness}
Fix an indecomposable homotopy class $\alpha \in \pi_2(\cN)$ and let $\Theta = \Theta(\alpha)$ be its lowest energy level (as in Definition \ref{df:notation}). We say that an energy minimizer $u$ is $\delta$-flat in the ball $\B_r(x)$ (of type $\alpha$) if 
\begin{enumerate}
\item
\label{df:flatness-energy}
$x$ is a singular point of $u$ and $\Theta \le \theta_u(x,0) \le \theta_u(x,r) \le \Theta + \delta$, 
\item
\label{df:flatness-Reifenberg}
$\sing u$ is $\frac{1}{10}$-Reifenberg flat in $\B_r(x)$ with respect to some $L$, and $u$ restricted to $(x+L^\perp) \cap \partial \B_r(x)$ has homotopy type $\alpha$. 
\end{enumerate}
\end{df}

From now on, we consider a non-trivial indecomposable class $\alpha$ and its lowest energy level $\Theta = \Theta(\alpha)$ to be fixed. 

The main feature of this definition is that $\delta$-flatness in a ball trivially ensures that condition \ref{df:flatness-energy} is satisfied in all smaller concentric balls, and one only needs to check condition \ref{df:flatness-Reifenberg} (see Corollary \ref{cor:all-scales}). This is why the constant $\tfrac{1}{10}$ in condition \ref{df:flatness-Reifenberg} was chosen reasonably weak. 

\section{Regularity of the singular set}
\label{ch:proof}

\subsection*{Stability of indecomposable singularities}

As a first step, we show that if an energy minimizer $u$ restricted to some sphere has homotopy type $\alpha$, then $\sing u$ satisfies the flatness condition of Definition \ref{df:flatness} and the energy density of $u$ cannot drop in a smaller ball. Note that the claim of Lemma \ref{lem:no-drop} is essentially stronger than the condition $L \cap \B_1 \subseteq \B_{\varepsilon} (\sing_{\ge \Theta} u)$ from Definition \ref{df:set-flatness}. 

\medskip

Observe that some tubular neighborhood $\B_{\eta} \cN \subseteq \R^M$ admits a continuous retraction $\pi_\cN$ onto $\cN$. As a consequence, if two continuous functions $f,g$ into $\cN \subseteq \R^M$ are close enough in supremum norm, then 
\[
(t,x) \mapsto \pi_{\cN}(tf(x) + (1-t)g(x)) 
\]
yields a homotopy between them. 

\begin{lem}
\label{lem:no-drop}
Assume that $\sing u \cap \B_1 \subseteq \B_{\varepsilon} L$ for some $0 < \varepsilon < \tfrac 12$ and some $(n-3)$-dimensional plane $L$ through $0$. Assume further that $u$ restricted to $L^\perp \cap \partial \B_1$ has homotopy type $\alpha$. Then 
\[
L \cap \B_{1-\varepsilon} \subseteq \pi_L( \sing_{\ge \Theta} u \cap \B_1 ),
\]
where $\pi_L$ denotes the orthogonal projection onto $L$. In particular, $\sing_{\ge \Theta} u$ is $\varepsilon$-Reifenberg flat in $\B_1$. 
\end{lem}

Before giving the full proof, let us consider the special case when $\Theta(\alpha)$ is the lowest among all non-trivial homotopy types. In this case, the proof is simpler and does not depend on the deep results of Naber and Valtorta \cite{NabVal17}.

For each $y \in L \cap \B_{1-\varepsilon}$, $u$ restricted to the sphere $(y + L^\perp) \cap \partial \B_\varepsilon(y)$ has homotopy type $\alpha$, hence cannot be continuously extended to the ball $(y + L^\perp) \cap \B_\varepsilon(y)$. This shows the weaker inclusion $L \cap \B_{1-\varepsilon} \subseteq \pi_L( \sing u \cap \B_1 )$. Recall that $\cH^{n-3}$-a.e. point $z \in \sing u$ belongs to $\sing_* u$ and hence $\theta_u(z,0) \ge \Theta$ due to our additional assumption. Since $\sing_{\ge \Theta} u$ is a closed set, we obtain the stronger inclusion. 

\begin{proof}[Proof of Lemma \ref{lem:no-drop}]
For simplicity, let us rotate so that $L = \R^{n-3} \times \mathbf{0}$. Assume for the contrary that $L \cap \B_{1-\varepsilon}$ is not covered by the projection of $\sing_{\ge \Theta} u \cap \B_1$. Since the latter is a compact set, it has to be disjoint with some cylinder $\B^{n-3}_\delta(z) \times \R^3$. 

Recall that by the recent important work of Naber and Valtorta \cite{NabVal17}, the set $\sing u \cap \B_1$ has finite $\cH^{n-3}$ measure. Moreover, the set $\sing u$ is $(n-3)$-rectifiable and for $\cH^{n-3}$-a.e. $y \in \sing u$ there exists an $(n-3)$-dimensional tangent plane $\Tan(\sing u,y)$ coinciding with $S(\varphi)$ for every tangent map $\varphi$ of $u$ at $y$. Let us temporarily assume that these tangent planes are transversal to $\mathbf{0} \times \R^{n-3}$, i.e. 
\begin{equation}
\label{eq:transverse}
\Tan(\sing u,y) \pitchfork \mathbf{0} \times \R^3 \quad \text{for $\cH^{n-3}$-a.e. } y \in \sing u \cap \B_1. 
\end{equation}

We shall need Eilenberg's inequality, a Fubini type inequality valid for any $\cH^{n-3}$-measurable set $A$ with finite measure (see \cite[7.7,~7.8]{Mat95}): 
\[
\int_{\B^{n-3}_\delta(z)} \cH^0(A \cap \pi_L^{-1}(y)) \dd y \le \omega_{n-3} \cH^{n-3}(A) 
\]
Applying the above inequality twice -- once with $A$ as the singular set and once with $A$ as its exceptional part of measure zero -- we learn that for a.e. $y \in \B^{n-3}_\delta(z)$ the slice $\sing u \cap \B_1 \cap \pi_L^{-1}(y)$ consists of finitely many points, at each of them the tangent plane exists and is transverse to $\mathbf{0} \times \R^3$ (i.e. the direction of slicing). 

Let us choose one such $y$ and denote these singular points by $p_1,\ldots,p_k$. Let also 
\[
L_j := \Tan(\sing u, p_j), 
\quad 
r_0 := \tfrac 12 \min_{i,j} \left ( |p_i-p_j|, \varepsilon - |p_i-y| \right ).
\]
For each $j=1,\ldots,k$, there is a HCM $\varphi_j \colon \R^n \to \cN$ with $S(\varphi_j) = L_j$ such that the sequence of rescaled maps $u_{r_i}(x) = u(p_j+r_i x)$ converges to $\varphi_j$ in $W^{1,2}(\B_1)$ for some sequence $r_i \to 0$. Note that by our assumption, $\varphi_j$ has energy density strictly less than $\Theta$. Since this convergence is uniform away from $L_j$, maps $u_{r_i}$ and $\varphi_j$ are homotopic on $L_j^\perp \cap \partial \B_1$ for large enough $i$. Tilting $L_j^\perp$ to $\mathbf{0} \times \R^3$ and rescaling, we get that maps $u$ and $\varphi_j$ are homotopic on $\pi_L^{-1}(y) \cap \partial \B_{r_j}(p_j)$ for some small $r_j < r_0$. Recalling that $u$ restricted to $\mathbf{0} \times \R^3 \cap \partial \B_1$ (and hence also to $\pi_L^{-1}(y) \cap \partial \B_\varepsilon(y)$) has homotopy type $\alpha$, we conclude 
\[
\alpha = [\varphi_1] + \ldots + [\varphi_k] \quad \text{in } \pi_2(\cN),
\]
where each $\varphi_j$ has energy density smaller than $\Theta$, which is a contradiction with the assumption that $\alpha$ is indecomposable. 

To finish the proof, we need to get rid of the additional assumption \eqref{eq:transverse}. This is done by using the following simple transversality lemma. 

\begin{lem}
\label{lem:transverse}
Let $n = a+b$, consider the Grassmannian $G(n,a)$ with the standard volume measure $\lambda$ and $G(n,b)$ with a finite positive Borel measure $\mu$. Then the set 
\[
\{ E \in G(n,a) : \mu(\{ F \in G(n,b) : E \not\pitchfork F \}) > 0 \})
\]
has zero $\lambda$ measure. 
\end{lem}

Postponing its proof for the moment, we complete the reasoning as follows. Choose $a = 3$, $b = n-3$, and let $\mu$ be the measure $\cH^{n-3} \llcorner \sing u \cap \B_1$ pushed-forward by the map $\Tan(\sing u, \cdot)$, i.e.
\[
\mu(U) = \cH^{n-3} (\{ y \in \sing u \cap \B_2 : \Tan(\sing u, y) \in U \}). 
\]
Then the set in Lemma \ref{lem:transverse} has measure zero, in particular its complement is dense. Hence we can choose $E \in G(n,3)$ so that $E \pitchfork F$ for $\mu$-a.e. $F \in G(n,n-3)$, with $E$ arbitrarily close to $\mathbf{0} \times \R^3$. This amounts to satisfying \eqref{eq:transverse} with a slightly tilted direction of slicing. Recall that $\sing_{\ge \Theta} u \cap \B_1$ is disjoint with the cylinder $\B^{n-3}_\delta(z) \times \R^3$, in consequence it is also disjoint with a smaller cylinder in direction $E^\perp$. It is easy to see that the rest of the proof remains unchanged. 
\end{proof}

\begin{proof}[Proof of Lemma \ref{lem:transverse}]
First note that for each $F \in G(n,b)$ the set of all $E \in G(n,a)$ non-transversal to $F$ is a finite sum of smooth submanifolds of $G(n,a)$ of positive codimension 
\[
\{ E \in G(n,a) : E \not\pitchfork F \} = \bigcup_{c=1}^{\min(a,b)} \{ E \in G(n,a) : \dim E \cap F = c \},
\]
hence it has zero $\lambda$ measure. Applying Fubini theorem, we get 
\begin{align*}
\int_{G(n,a)} \mu(\{ F \in G(n,b) : E \not\pitchfork F \}) \dd \lambda(E) 
& = \int_{G(n,a)} \int_{G(n,b)} \mathbf{1}_{E \not\pitchfork F} \dd \mu(F) \dd \lambda(E) \\
& = \int_{G(n,l)} \lambda(\{ E \in G(n,a) : E \not\pitchfork F \}) \dd \mu(F) \\
& = 0, 
\end{align*}
so the integrand has to be zero for $\lambda$-a.e. $E \in G(n,a)$. 
\end{proof}

\subsection*{Propagation of $\delta$-flatness to finer scales}

In this section we investigate some important consequences of Definition \ref{df:flatness}. Assuming that an energy minimizing map $u$ is $\delta$-flat in $\B_1$ (with small $\delta > 0$), we shall see that $\sing u$ is actually more flat than a priori assumed (Lemma \ref{lem:fund}), $u$ is also $\delta$-flat in all smaller concentric balls (Corollary \ref{cor:all-scales}), and that $0 \in \sing_\alpha u$ (Corollary \ref{cor:top-stratum}). 

\begin{lem}
\label{lem:fund}
For every $\varepsilon > 0$ there is $\delta_1(\varepsilon) > 0$ such that if $u$ is $\delta_1$-flat in $\B_1$, then $\sing u$ is $\varepsilon$-Reifenberg flat in $\B_1$ and $\| u - \varphi \|_{W^{1,2}(\B_1)} \le \varepsilon$ for some HCM $\varphi$ of homotopy type $\alpha$ with energy density $\Theta$. Moreover, $\sing u$ is $\varepsilon$-Reifenberg flat in $\B_1$ with respect to the $(n-3)$-dimensional plane $S(\varphi)$. 
\end{lem}

\begin{proof}
We employ the usual contradiction argument. Let $u_k$ be a sequence of minimizing harmonic maps such that $u_k$ is $1/k$-flat in $\B_1$, with $\sing u$ $\tfrac{1}{10}$-Reifenberg flat with respect to a fixed plane $L$. 
Choosing a subsequence, we have $u_k \to \varphi$ in $W^{1,2}(\B_1)$ for some energy minimizing $\varphi$. By condition \ref{df:flatness-energy} in Definition \ref{df:flatness}, $\varphi$ is homogeneous with energy density $\Theta$. By Lemma \ref{lem:no-drop}, for each $k$ the set $\sing_{\ge \Theta} u_k$ is $\frac{1}{10}$-Reifenberg flat in $\B_1$ with respect to $L$. Taking the limit and exploiting the upper semicontinuity of $\theta_\cdot(\cdot,0)$ with respect to both the map and the point, we conclude that the set 
\[
S(\varphi) \equiv \sing_{\ge \Theta} \varphi
\]
is not contained in any $(n-4)$-dimensional plane. On the other hand, it is itself a~linear subspace of dimension at most $n-3$, so we learn that $\varphi$ is a HCM of homotopy type $\alpha$ (by uniform convergence away from $L$). For large enough $k$, $u_k$ is $\varepsilon$-close to $\varphi$ in $W^{1,2}(\B_1)$ and its singular set is contained in $\B_\varepsilon S(\varphi)$ (this is a~consequence of upper semicontinuity of $\theta_\cdot(\cdot,0)$ and $\varepsilon$-regularity \eqref{eq:eps-regularity}), which finishes the proof by another application of Lemma \ref{lem:no-drop}. 
\end{proof}

\begin{cor}
\label{cor:all-scales}
If $\delta \le \delta_1(\frac{1}{20})$ and $u$ is $\delta$-flat in $\B_1$, then $u$ is also $\delta$-flat in any smaller ball $\B_r$ centered at $0$ with $0 < r \le 1$. 
\end{cor}

\begin{proof}
Condition \ref{df:flatness-energy} of Definition \ref{df:flatness} is trivially satisfied. As for condition \ref{df:flatness-Reifenberg}, it follows from Lemma \ref{lem:fund} that $\sing u$ is $\frac{1}{20}$-Reifenberg flat in $\B_1$, hence $\frac{1}{10}$-flat in any ball $\B_r$ with $\frac 12 \le r \le 1$. In consequence, $u$ is $\delta$-flat in each of these balls. Then the claim follows by iteration of Lemma \ref{lem:fund} rescaled to smaller and smaller balls. 
\end{proof}

\begin{cor}
\label{cor:top-stratum}
If $\delta \le \delta_1(\frac{1}{20})$ and $u$ is $\delta$-flat in $\B_1$, then every tangent map to $u$ at $0$ is a HCM of type $\alpha$. In particular, $0 \in \sing_{\alpha} u$. 
\end{cor}

\begin{proof}
Let $\varphi$ be any tangent map to $u$ at $0$, i.e. a $W^{1,2}(\B_1)$-limit of rescaled functions $u_k(x) = u(r_k x)$ for some sequence $r_k \to 0$; any thus obtained $\varphi$ is homogeneous. By Corollary \ref{cor:all-scales}, each $u_k$ is $\delta$-flat in $\B_1$ (with $\sing u_k$ $\tfrac{1}{10}$-Reifenberg flat with respect to some $L_k$), so the claim follows from Lemma \ref{lem:no-drop} as in the proof of Lemma \ref{lem:fund}. The only difference is that the planes $L_k$ may change, but without loss of generality $L_k \to L$ in $G(n,n-3)$, which is enough to conclude that $\sing_{\ge \Theta} \varphi$ spans an $(n-3)$-dimensional plane. 
\end{proof}

\subsection*{Moving the ball center}

\begin{prop}
\label{prop:all-balls}
For every $\varepsilon > 0$ there is $\delta_2(\varepsilon) > 0$ such that if $u$ is $\delta_2$-flat in $\B_2$, then $u$ is $\delta_1(\varepsilon)$-flat in each of the balls $\B_r(z)$ with $z \in \sing_{\ge \Theta} u \cap \B_1$ and $0 < r \le 1$. Moreover, the sets $\sing_\alpha u$ and $\sing_{\ge \Theta} u$ restricted to the ball $\B_1$ coincide. 
\end{prop}

\begin{proof}
Choose $\delta_2 := \min(\delta_1(\varepsilon),\delta_1(\eta/2))$ according to Lemma \ref{lem:fund}, where $\eta>0$ is to be fixed in a moment. Applying Lemma \ref{lem:fund} rescaled to the ball $\B_2$, denote by $\varphi$ the approximating HCM and let $L = S(\varphi)$. To obtain the first claim, we first show that $\theta_u(z,1) \le \Theta + \delta_1(\varepsilon)$ for each $z \in \B_1 \cap \B_{\eta} L$. First, 
\[
\int_{\B_1(z)} |\nabla u|^2 \le \int_{\B_1(z)} |\nabla \varphi|^2 + C \eta^2,  
\]
by Lemma \ref{lem:fund}. If $z' = \pi_L(z)$, then $|z-z'| < \eta$ and 
\[
\int_{\B_1(z)} |\nabla \varphi|^2 \le \int_{\B_{1+\eta}(z')} |\nabla \varphi|^2 = (1+\eta)^{n-2} \Theta
\]
by $L$-invariance of $\varphi$ in $z'$-direction. If $\eta$ is chosen small enough (depending on $\delta_1(\varepsilon)$), we obtain $\theta_u(z,1) \le \Theta + \delta_1(\varepsilon)$. 

\medskip

Since each point $z \in \sing_{\ge \Theta} u \cap \B_1$ lies in $\B_\eta L$ by Lemma \ref{lem:fund}, the above reasoning shows that condition \ref{df:flatness-energy} of Definition \ref{df:flatness} holds for the ball $\B_1(z)$. Condition \ref{df:flatness-Reifenberg} is satisfied by our assumptions, this ball is $\delta_1(\varepsilon)$-flat. Then Corollary \ref{cor:all-scales} implies $\delta_1(\varepsilon)$-flatness of $u$ also in all smaller balls $\B_r(z)$. 

By Corollary \ref{cor:top-stratum} we now have $z \in \sing_\alpha u$ for each $z \in \sing_{\ge \Theta} \cap \B_1$. The inverse inclusion $\sing_\alpha u \subseteq \sing_{\ge \Theta} u$ is evident from the definition of $\Theta(\alpha)$. 
\end{proof}

\begin{cor}
\label{cor:only-top}
Under the assumptions of Proposition \ref{prop:all-balls}, the whole singular set $\sing u$ restricted to the ball $\B_{1/2}$ coincides with $\sing_{\ge \Theta} u$ (and hence with $\sing_\alpha u$). 
\end{cor}

\begin{proof}
Assume that the ball $\B_{1/2}$ contains a point $p \in \sing u \setminus \sing_{\ge \Theta} u$. We may choose a point $z \in \sing_{\ge \Theta} u$ closest to $p$ (as it is a closed set) and set $r = 2|p-z|$. Clearly $z \in \B_1$ and $0 < r \le 1$, so $u$ is $\delta_1(\varepsilon)$-flat in $\B_r(z)$. Choose $L = L(z,r)$ according to Definition \ref{df:set-flatness}. Then by Lemma \ref{lem:no-drop} there is a point $z' \in \sing_{\ge \Theta} u \cap \B_r(z)$ such that $\pi_L(z')=\pi_L(p)$. Since both $|\pi_L(p)-p|$ and $|\pi_L(z')-z|$ are less than $\tfrac{r}{10}$, the triangle inequality yields a~contradiction with minimality of $z$. 
\end{proof}

In order to apply the above results, one needs to know that $u$ is $\delta$-flat in at least one ball. 

\begin{lem}
\label{lem:some-flatness}
Let $\delta > 0$. If $0 \in \sing_\alpha u$ and $\theta_u(0,0) < \Theta + \delta$, then there is $r > 0$ such that $u$ is $\delta$-flat in $\B_r$. 
\end{lem}

\begin{proof}
Note that condition \ref{df:flatness-energy} of Definition \ref{df:flatness} is trivially satisfied for small enough~$r$. 

By definition of $\sing_\alpha u$, some sequence of rescaled functions $u_k(x) = u(r_k x)$ converges in $W^{1,2}(\B_1)$ to a HCM $\varphi$ of homotopy type $\alpha$ for some sequence $r_k \to 0$. For large enough $k$, we have $\sing u_k \cap \B_1 \subseteq \B_{1/10} S(\varphi)$. Since the convergence is uniform away from $S(\varphi)$, $u_k$ restricted to $S(\varphi)^\perp \cap \partial \B_1$ has homotopy type $\alpha$, so condition \ref{df:flatness-Reifenberg} follows from Lemma \ref{lem:no-drop}. Rescaling, we see that $u$ is $\delta$-flat in $\B_{r_k}$ for large enough $k$. 
\end{proof}

\begin{rem}
\label{rem:unique-type}
Combining Lemma \ref{lem:some-flatness} with Corollary \ref{cor:top-stratum}, we see that \textit{some} can be changed to \textit{any} in the definition of $\sing_\alpha$, if only we restrict ourselves to points with energy density close to optimal. 
That is, if $y \in \sing_\alpha u$ and $\theta_u(y,0) < \Theta + \delta_1(\tfrac{1}{20})$, then every tangent map of $u$ at $y$ is a HCM of type $\alpha$.
\end{rem}

We are now ready to prove the main theorem. 

\begin{proof}[Proof of Theorem \ref{thm:Holder-regularity}]
Fix the H{\"o}lder exponent $0 < \gamma < 1$ and choose $\varepsilon = \varepsilon(\gamma,n) > 0$ according to Reifenberg's topological disc theorem (Theorem \ref{thm:Reifenberg}), then fix $\delta$ to be $\delta_2(\varepsilon)$ from Proposition \ref{prop:all-balls}. 

Choose a point $p \in \sing_\alpha u$ such that $\theta_u(p,0) < \Theta + \delta$. According to Lemma \ref{lem:some-flatness}, $u$ is $\delta_2(\varepsilon)$-flat in some ball $\B_{2r}(p)$. By Proposition \ref{prop:all-balls}, the set $\sing_\alpha u \cap \B_{r}(p)$ is closed and $\varepsilon$-flat in each ball $\B_s(z)$ centered at $z \in \sing_\alpha u \cap \B_{r}(p)$ with radius $0 < s < r$. Applying Theorem \ref{thm:Reifenberg}, we conclude that $\sing_\alpha u \cap \B_{r}(p)$ is bi-H{\"o}lder equivalent (with exponent $\gamma$) to an $(n-3)$-dimensional ball. 

By upper semicontinuity of $\theta_u(\cdot,0)$ we can ensure $\theta_u(y,0) < \Theta + \delta$ for all $y \in \B_{r}(p)$ (jsut by taking $r$ small enough), which together with Corollary \ref{cor:only-top} shows that the set in question forms an open subset of $\sing u$. 
\end{proof}

\section{Additional results}
\label{ch:additional}

In this subsection we discuss two elementary observations that give a better description of $\delta$-flatness, but were not needed for the proof of Theorem \ref{thm:Holder-regularity}. We fix an indecomposable homotopy type $\alpha \in \pi_2(\cN)$ and its lowest energy level $\Theta = \Theta(\alpha)$. 

\medskip

The following lemma shows that condition \ref{df:flatness-Reifenberg} in Definition \ref{df:flatness} can be dropped if one assumes a priori that $x \in \sing_\alpha u$. This gives us an equivalent condition for $\delta$-flatness. 

\begin{lem}
\label{lem:only-energy-assumption}
Assume that $0 \in \sing_\alpha u$. If $\delta \le \delta_1(\frac{1}{20})$ and $\theta_u(0,1) \le \Theta + \delta$, then $u$ is $\delta$-flat in $\B_1$. 
\end{lem}

\begin{proof}
Inspecting the proof of Lemma \ref{lem:fund}, we see that condition \ref{df:flatness-Reifenberg} of Definition \ref{df:flatness} was only needed to ensure required symmetry of approximating homogenous minimizer $\varphi$. Hence it would be enough to assume condition \ref{df:flatness-Reifenberg} of Definition \ref{df:flatness} in a~smaller ball $\B_{1/2}$, and $\delta$-flatness in $\B_1$ follows as in Lemma \ref{lem:fund}. 

By Lemma \ref{lem:some-flatness}, there is $r > 0$ (possibly very small) such that $u$ is $\delta$-flat in $\B_r$. Applying the reasoning above, we see it is also $\delta$-flat in every ball $\B_s$ with $r \le s \le \min(1,2r)$. An iteration of this argument (as in Corollary \ref{cor:all-scales}, but in the opposite direction) leads to the claim. 
\end{proof}

\medskip

The last lemma gives a uniform bound (independent of $u$) for the rate of convergence $\theta_u(x,r) \to \theta_u(x,0)$ when $r \to 0$, assuming $\theta_u(x,r)$ is already close to $\theta_u(x,0)$. This assumption cannot be dropped, if only there exist tangent maps $\varphi \colon \R^n \to \cN$ with $\dim_H \sing \varphi = n-3$ which are not HCMs. 

An additional assumption is needed to ensure that the energy density is not greater than $\Theta$. This assumption is automatically satisfied if $\cN$ is real-analytic of integrable in the sense of \cite[Ch.~3.13]{Sim96}; see the remark preceding Corollary \ref{cor:discrete-levels}. 

\begin{lem}
\label{lem:rate-of-convergence}
Assume additionally that $\Theta$ is an isolated energy level for HCMs of type $\alpha$. Assume that $0 \in \sing_\alpha u$ and $\theta_u(0,1) \le \Theta + \delta_4$ with $\delta_3(n,\alpha,\cN) > 0$ sufficiently small. Then for every $\delta > 0$ there is $r(\delta,n,\cN) > 0$ such that $\theta_u(0,r) \le \Theta + \delta$ (in consequence, $u$ is $\delta$-flat in $\B_r$). 
\end{lem}

\begin{proof}
We choose $\delta_3 > 0$ smaller than $\delta_1(\frac{1}{20})$ from Lemma \ref{lem:fund} and such that 
\[
\int_{\B_1^n} |\nabla \varphi|^2 \notin (\Theta,\Theta+\delta_3]
\]
for each HCM $\varphi$ of type $\alpha$. 

For the sake of contradiction, assume there is a sequence of such energy minimizing maps $u_k$ with 
\[
\Theta + \delta \le \theta_{u_k}(0,1/k) \le \theta_{u_k}(0,1) \le \Theta + \delta_3.
\]
Taking a subsequence, we obtain a limit map $u$ such that 
\[
\Theta + \delta \le \theta_u(0,0) \le \theta_{u}(0,1) \le \Theta + \delta_3.
\]
It follows from Lemma \ref{lem:only-energy-assumption} that each $u_k$ is $\delta_1(\frac{1}{20})$-flat in $\B_1$, hence so is $u$ and by Corollary \ref{cor:top-stratum} we infer $0 \in \sing_\alpha u$. In particular the energy density $\theta_u(0,0)$ is either $\Theta$ or greater than $\Theta+\delta_3$, a contradiction.  
\end{proof}

\section*{Acknowledgments}

The author would like to thank Maciej Borodzik for fruitful discussions and helpful suggestions. 

\bibliography{harmonic-maps}
\bibliographystyle{acm}

\end{document}